\documentclass[11pt,reqno]{amsart} 
\usepackage[utf8]{inputenc} 
\usepackage{color}
\usepackage{cite} 
\usepackage{esint}
\usepackage[english]{babel}
\usepackage{amsmath}
\usepackage{amsthm}
\usepackage{amsfonts,mathrsfs,relsize,bigints,stmaryrd}
\usepackage{amssymb}
\usepackage{anysize}
\usepackage{hyperref}
\usepackage{appendix}
\usepackage{mathtools}
\usepackage{graphicx}
\usepackage{xcolor}
\usepackage[normalem]{ulem}
\makeatletter  \@addtoreset{equation}{section} \makeatother

\usepackage[left=2.65cm,right=2.65cm,top=2.7cm,bottom=2.7cm]{geometry}

\hypersetup{
	colorlinks=true,
	linkcolor=blue, 
	citecolor=blue,
	filecolor=blue,
	urlcolor=blue,
}

\newtheorem{lem}{Lemma}[section]
\newtheorem{theo}{Theorem}[section]

\newtheorem{rem}{Remark}[section]

\newcommand\minus\backslash

\newcommand\lan\langle
\newcommand\ran\rangle

\renewcommand\leq\leqslant
\renewcommand\geq\geqslant

\DeclareMathOperator{\N}{\mathbb{N}}
\DeclareMathOperator{\R}{\mathbb{R}}

\DeclareMathOperator{\T}{\mathbb{T}}
\DeclareMathOperator{\D}{\mathbb{D}}

\allowdisplaybreaks 

\begin{document}

\title[Survey on periodic vortex patches]{Survey on periodic vortex patches}

\author[C. Garc\'ia]{Claudia Garc\'ia}
\address{ Departamento de Matem\'atica Aplicada \& Research Unit ``Modeling Nature'' (MNat), Facultad de Ciencias, Universidad de Granada, 18071 Granada, Spain}
\email{ claudiagarcia@ugr.es}



\begin{abstract}
This survey revisits classical results on the existence of periodic solutions in incompressible fluid dynamics. Owing to the breadth of the subject, we restrict our attention to the simplest and most illustrative framework: the two-dimensional Euler equations and vortex patch solutions. The aim is not to present new results, but rather to provide a unified and streamlined exposition of several classical constructions by combining ideas from the original works with more recent developments.
\end{abstract}

\maketitle


\section{Introduction}\label{sec:intro}

The 2D incompressible Euler equations in velocity-pressure formulation read as
\begin{equation}\label{intro:systemeuler}\left\{
\begin{array}{ll}
    \partial_t u+(u\cdot \nabla)u+\nabla p=0, & \textnormal{in }[0,T)\times \R^2,\\
    \nabla \cdot u=0,
    &\textnormal{in }[0,T)\times \R^2,\\
    |u(t,x)|\rightarrow 0, &\textnormal{as }|x|\rightarrow +\infty,\\
    u(0,\cdot)=u_0,& \textnormal{in } \R^2,
\end{array}\right.
\end{equation}
where $u=u(t,x)=(u_1(t,x), u_2(t,x))$ is the velocity field of the fluid, $p=p(t,x)$ is the pressure and $t\in[0,T)$, $x=(x_1,x_2)\in\R^2$. That is a system of three equations for three scalar quantities: $u_1, u_2, p$. However, $p$ can be recovered easily by applying the $\nabla\cdot$ operator in the first equation, achieving that $p$ is the solution of the following elliptic equation:
$$
\Delta p=-\nabla\cdot (u\cdot \nabla)u.
$$
In this survey, we will be interested in the vorticity of the fluid which, in the two-dimensional setting, reads as
$$
\omega=\nabla^\perp \cdot u,\quad \nabla^\perp:=(-\partial_{x_2},\partial_{x_1}).
$$
Applying the differential operator $\nabla^\perp\cdot$ to \eqref{intro:systemeuler} we arrive to the 2D incompressible Euler equations in velocity-vorticity formulation:
\begin{equation}\label{intro:systemeuler2}\left\{
\begin{array}{ll}
    \partial_t \omega+u\cdot \nabla\omega =0, & \textnormal{in }[0,T)\times \R^2,\\
    \nabla \cdot u=0,
    &\textnormal{in }[0,T)\times \R^2,\\
     \omega=\nabla^\perp\cdot u,
    &\textnormal{in }[0,T)\times \R^2,\\
    |u(t,x)|\rightarrow 0, &\textnormal{as }|x|\rightarrow +\infty,\\
    \omega(0,\cdot)=\omega_0,& \textnormal{in } \R^2.
\end{array}\right.
\end{equation}

The three equations in the middle of \eqref{intro:systemeuler2} can be solved via the well-known Biot-Savart law, obtaining the following equivalent formulation:
\begin{equation}\label{intro:systemeuler3}\left\{
\begin{array}{ll}
    \partial_t \omega+u\cdot \nabla\omega=0 , & \textnormal{in }[0,T)\times \R^2,\\
   u=\nabla^\perp \psi,
    &\textnormal{in }[0,T)\times \R^2,\\
    \psi=G\star\omega,
    &\textnormal{in }[0,T)\times \R^2,\\
    \omega(0,\cdot)=\omega_0,& \textnormal{in } \R^2,
\end{array}\right.
\end{equation}
where $\psi=\psi(t,x)$ is called the stream function of the fluid and $G$ is the fundamental solution of the Laplacian in $\R^2$:
$$
G(x):=\frac{1}{2\pi}\ln|x|.
$$
Notice that, in the two-dimensional setting, the vorticity $\omega$ satisfies a transport equation. However, this equation is both nonlinear (since the velocity field $u$ depends on $\omega$) and nonlocal (since the stream function $\psi$ depends nonlocally on $\omega$ through a convolution operator). Moreover, since $u$ is a divergence-free vector field and $\omega$ satisfies a transport equation, all $L^p$ norms of the vorticity are conserved.

In particular, the classical Yudovich theory \cite{Y63} ensures global existence and uniqueness of solutions when the initial data $\omega_0$ belongs to $L^1\cap L^\infty$. All the solutions considered in this survey lie in this class. Consequently, our focus here is not on questions of existence or uniqueness, but rather on the identification and analysis of coherent structures in the flow. More specifically, we will study vortex patches, namely vorticity distributions whose support is a bounded domain:
$$
\omega_0=\sum_{i=1}^N \varpi_i {\bf 1}_{D_i},
$$
where $\varpi_i\in\R$ is the vorticity strength in each bounded domain $D_i$. Note that thanks to Yudovich theory we know that such initial condition, which belongs to $L^1\cap L^\infty$, has a unique associated global solution $\omega(t,\cdot)$. Indeed, the patch structure is preserved by the system due to the transport equation obtaining that:
$$
\omega(t,\cdot)=\sum_{i=1}^N \varpi_i {\bf 1}_{D_{i}(t)},
$$
where the domain $D_i$ evolves with time to $D_{i}(t)$, and $D_{i}(t)$ has the same area due to the incompressibility condition. Hence, studying vortex patch solutions is equivalent to study the evolution of the boundary of $D_{i}(t)$, which can be done by looking at the well-known contour dynamics equation:
$$
\partial_t \gamma_i(t,\theta)=\sum_{j=1}^N\frac{\varpi_j}{2\pi}\int_0^{2\pi}\ln(|\gamma_i(t,\theta)-\gamma_j(t,\eta)|)\partial_\eta \gamma_j(t,\eta)d\eta,
$$
where $\gamma_i(t,\cdot)$ is a parametrization of the boundary of $D_{i}(t)$. The boundary regularity of the patch is preserved in $C^{k,\alpha}$ for $k\in\N$ and $\alpha\in(0,1)$, see \cite{chemin, bertozzi-constantin, serfati}, and  the $C^k$ regularity, for $k\geq 2$, is ill posed, see \cite{kiselev-luo}.

A particular feature of the two-dimensional Euler system is the existence of a large family of stationary solutions. Indeed, if $\psi_0$ satisfies
$$
\Delta \psi_0=f(\psi_0),
$$
for some smooth function $f$, then the corresponding vorticity generates a stationary solution of the system. A notable class of solutions to this elliptic equation is given by radial functions. In particular, if $\omega_0\in L^1\cap L^\infty$ is radial, then it produces a stationary solution of the two-dimensional Euler equations. We will refer to this family of stationary solutions as {\it trivial solutions}.

Although our main interest in this survey is not the study of stationary solutions themselves, but rather periodic solutions bifurcating from them, there exists an extensive literature devoted to the construction of nontrivial stationary solutions; see \cite{Nadirashvili, GSPSY-rigidity, GSPS:2021} and the references therein. In the vortex patch setting, it is known from \cite{F00} that the only stationary simply-connected patch is radial one, namely the disc $\D$ (classically called the {\it Rankine vortex}). We refer also to the stationary annulus patch, which was studied in \cite{HHMV16, HMV15}.

In this survey, we establish the existence of periodic solutions by exploiting the symmetries of the trivial stationary patch $\omega_0={\bf 1}_{\D}$. For clarity of exposition, we divide the introduction into three main topics: i) rigid periodic simply connected vortex patches, ii) rigid periodic disjoint vortex patches, and iii) non--rigid periodic vortex patches. The literature on this subject is vast, and here we restrict our attention to results concerning the two-dimensional Euler equations and vortex patch solutions. Extensions to other fluid models and to other types of solutions have also been developed; see \cite{ADPM21,CQZZ21,CQZZ22,CWWZ21,CCGS16,CCGS19,CCGS20,HHH16,DHR19,G20,G21,GH22,GHM22,GHM22-1,GHS20,GCGS20,HH15,HH21,HMW20,HW22,HM16-1,HM16,HM17,HMV15,HXX22,R22,R21, CL26, CCG16, HR22, HR21, HHR23, GHM23, GSPSY-sheets, GHR25} and references therein.

Before reviewing each of the three topics mentioned above, let us briefly discuss the main techniques used in this area. Roughly speaking, as we shall see later, most works on the existence of periodic solutions near a trivial one rely on perturbative arguments. In general, the problem can be formulated in terms of a suitable—often nonlinear and nonlocal—functional $F:\R\times X\rightarrow Y$, where $X$ and $Y$ are Banach spaces. The problem then reduces to finding nontrivial zeros of this functional:
$$
(\lambda, f)\textnormal{ s.t. } F(\lambda,f)=0,
$$
in a neighborhood of a trivial solution (or family of trivial solutions). One can then apply classical tools from nonlinear analysis to establish the existence of such solutions. The first fundamental result that can be used in this context is the well-known Implicit Function theorem for Banach spaces:
\begin{theo}[Implicit Function theorem]\label{th:IFT}
Let $F:\R\times \mathcal{X}\rightarrow \mathcal{Y}$, where $\mathcal{X}\subset X$ and $\mathcal{Y}\subset Y$ are subspaces of the Banach spaces $X$ and $Y$. Assume that there is $(\lambda_0,0)\in \R\times\mathcal{X}$ with the following properties:
\begin{enumerate}
    \item $F(\lambda_0,0)=0$,
    \item $F$ and $\partial_f F$ are continuous,
    \item $\partial_f F(\lambda_0,0):\mathcal{X}\rightarrow \mathcal{Y}$ is an isomorphism.
\end{enumerate}
Then there exists a neighborhood $U=I\times \mathcal{X}_0\subset \R\times\mathcal{X}$ of $(\lambda_0,0)$ and mapping $f:I\rightarrow \mathcal{X}_0$ such that $f(\lambda_0)=0$ and 
$$
F^{-1}(\{0\})\cap U=\{(\lambda,f(\lambda)), \quad \lambda\in I\}.
$$
\end{theo}
We remark that, for simplicity, we assume that the trivial root is $(\lambda_0,0)$. More generally, one could consider a trivial root of the form $(\lambda_0,f_0)$ and reduce the problem to the previous case by a suitable change of variables.

However, the Implicit Function theorem is not useful if there exists a trivial branch of solutions of the form
\[
F(\lambda,0)=0, \qquad \forall \lambda.
\]
Indeed, if the hypotheses of the Implicit Function theorem were satisfied at $(\lambda_0,0)$, it would follow that, in a neighborhood of $(\lambda_0,0)$, the only solutions are the trivial ones $(\lambda,0)$. Therefore, in order to find nontrivial solutions, the isomorphism condition in the Implicit Function theorem must fail (in particular, the kernel of the linearized operator should be nontrivial).

This is precisely the setting where bifurcation theory becomes relevant. In this survey we will make use of the classical Crandall--Rabinowitz theorem. In order to state it, let us first recall that a bounded linear operator $\mathcal{L}:\mathcal{X}\to\mathcal{Y}$ is called a Fredholm operator of index $i$ if
\begin{enumerate}
\item $\textnormal{dim (Ker }\mathcal{L})<\infty$,
\item $\textnormal{codim (Range }\mathcal{L}):=\textnormal{dim }(\mathcal{Y}/\textnormal{Range }(\mathcal{L}))<\infty$,
\end{enumerate}
and the index $i$ is given by the following integer
$$
i:=\textnormal{dim Ker }\mathcal{L}-\textnormal{codim Range }\mathcal{L}.
$$
Finally, let us announce the classical Crandall-Rabinowitz theorem, which can be found in \cite{K04}.
\begin{theo}[Crandall-Rabinowitz theorem]\label{th:CR}
	Let $F:\R\times \mathcal{X}\rightarrow \mathcal{Y}$, where $\mathcal{X}\subset X$ and $\mathcal{Y}\subset Y$ are subspaces of the Banach spaces $X$ and $Y$. Assume that there is $(\lambda_0,0)\in \R\times\mathcal{X}$ with the following properties:
		\begin{enumerate}
			\item $F(\lambda,0)=0$ for all $\lambda\in\mathbb{R}$.
			\item The partial derivatives  $\partial_\lambda F$, $\partial_f F$ and  $\partial_{\lambda}\partial_f F$ exist and are continuous.
			\item Denoting $\mathcal{L}(\lambda):=\partial_f F(\lambda,0)$, one has that the operator $\mathcal{L}(\lambda_0)$ is Fredholm of zero index and $\textnormal{Ker}\big(\mathcal{L}(\lambda_0)\big)=<f_0>$ is one-dimensional. 
			\item  Transversality assumption: $\mathcal{L}'(\lambda_0)[f_0] \notin \textnormal{Range}\big(\mathcal{L}(\lambda_0)\big)$.
		\end{enumerate}
		If $Z$ is any complement of  $\textnormal{Ker}\big(\mathcal{L}(\lambda_0)\big)$ in $\mathcal{X}$, then there is a neighborhood $U$ of $(\lambda_0,0)$ in $\mathbb{R}\times \mathcal{X}$, an interval  $(-a,a)$ with $a>0$, and two continuous functions $\Phi:(-a,a)\rightarrow\mathbb{R}$, $\beta:(-a,a)\rightarrow Z$ such that $\Phi(0)=\lambda_0$ and $\beta(0)=0$ and
		$$F^{-1}(\{0\})\cap U=\big\{\big(\Phi(s), s f_0+s\beta(s)\big) : |s|<a\big\}\cup\big\{(\lambda,0): (\lambda,0)\in U\big\}.$$
\end{theo}
We remark that in some situations the Crandall--Rabinowitz theorem cannot be applied directly, since its hypotheses are not satisfied. Nevertheless, one may still have some hope of finding nontrivial solutions. When the linearized operator is Fredholm but its kernel is not one-dimensional, or the codimension of its range is not equal to one, one may perform the classical Lyapunov--Schmidt reduction in order to search for nontrivial roots (recall that the proof of the Crandall--Rabinowitz theorem itself relies on this reduction). If the transversality condition, namely hypothesis~(iv), is not satisfied, it is sometimes possible to proceed by analyzing higher-order derivatives of the functional. The most delicate situation arises when the Fredholm property of the linearized operator fails, which may occur when the codimension of the range is infinite. This difficulty appears in the study of non-rigid periodic vortex patches due to the presence of small divisors and the resulting loss of derivatives in the inverse of the linearized operator, as we will explain at the end of this introduction. In such cases, one typically needs to implement a Nash--Moser scheme to construct the desired solutions.

\medskip\noindent
$\diamond$ {\it \sc Rigid periodic simply connected vortex patches:}

In this part, we shall focus on the existence of rigidly rotating patches with some angular velocity $\Omega$. More precisely, if one assumes that the vorticity is a pure rotation of the initial one, that is,
$$
\omega(t,x)=\omega_0(R_{-\Omega t}x),\quad R_{-\Omega t}:=\left(\begin{array}{cc}
    \cos(\Omega t) & \sin(\Omega t) \\
    -\sin(\Omega t) & \cos(\Omega t)
\end{array}\right),
$$
then the 2D Euler equations reduce to the following nonlocal and nonlinear stationary equation for the initial data $\omega_0$, and its associated $u_0$, as:
$$
(u_0(x)-\Omega x^\perp)\cdot \nabla\omega_0(x)=0, \quad x\in\R^2.
$$
The time dependence disappears due to the symmetries of the equation. Similar situation happens for translating solutions, i.e. $\omega(t,x)=\omega_0(x-Vt)$ with $V\in\R^2$, arriving to
$$
(u_0(x)-V)\cdot \nabla\omega_0(x)=0,\quad x\in\R^2,
$$
although here we shall focus on rotating solutions. Hence, inserting the patch ansatz into the rotating solutions equation one arrives to
\begin{equation}\label{eq:rotating}
(u_0(x)-\Omega x^\perp)\cdot n_i(x)=0, \quad x\in\partial D_i,
\end{equation}
for any $i=1,\dots , N$, and where $n_i$ is a normal vector to $\partial D_i$.

Here we focus on a single rotating vortex patch, often referred to as {\it V-states}. As already observed by Rankine in 1858, a disc and more generally a radially symmetric domain give rise to a V-state rotating with arbitrary angular velocity $\Omega\in\mathbb{R}$. Later, Kirchhoff \cite{K74} discovered nontrivial explicit examples of V-states with elliptical shapes.

In a neighborhood of the unit disc, such solutions were first obtained numerically by Deem and Zabusky \cite{DZ78}, and later constructed analytically by Burbea \cite{B82} using bifurcation techniques. More precisely, Burbea constructed a countable family of local curves of V-states with $m$-fold symmetry (that is, invariance under rotations of angle $\tfrac{2\pi}{m}$) bifurcating from the disc at the angular velocities
\[
\Omega_m=\frac{m-1}{2m}.
\]
This result was later revisited by Hmidi, Mateu, and Verdera \cite{HMV13}, who revisited the proof using the classical Crandall--Rabinowitz theorem together with conformal mapping techniques. They also established additional qualitative properties of the bifurcating patches, such as their regularity and convexity. In this survey, Burbea's result will be revisited again in Section~\ref{sec:rigidpatch}, where we present a simplified proof that avoids the use of conformal mappings, following instead the approach of more recent related works \cite{BHM22, HHM23, HHR23}.

The global continuation of Burbea's bifurcation branches was later obtained in \cite{HMW20} using global bifurcation techniques. Similar analyses have been carried out for the bifurcation from the annulus \cite{HHMV16} and from the Kirchhoff ellipse \cite{K74}. We also remark that in \cite{HM16} the authors treat the case where the transversality condition fails. Finally, rotating solutions without the patch structure have also been studied in \cite{GHS20, GHM23, CCGS19, CL26} by bifurcation from radial vorticities.

\medskip\noindent
$\diamond$ {\it\sc  Rigid periodic disjoint vortex patches}

The existence of periodic disjoint vortex patches is of great interest since it can be observed in different situations. One of the most interesting is the {\it K\'arm\'an Vortex Street}, which consist of two rows of {\it patches} of opposite vorticity strength  that travels with a common speed. That can be observed when a fluid with a preferred direction passes a body, for example, an island in the ocean.

Mathematically, the construction of such configurations is rather delicate with the techniques currently available. In order to apply perturbative arguments, it is necessary to start from a trivial solution. However, there is no obvious trivial branch of separated non concentrated finite patch configuration of the type needed for the perturbative construction. To overcome this difficulty, one typically considers the classical $N$-vortex problem, which arises formally in the limit where each patch shrinks to a point. In this regime, the characteristic functions describing the patches are formally replaced by Dirac delta distributions:
$$
\omega(t,\cdot)=\sum_{i=1}^N \varpi_i \delta_{z_{i}(t)}.
$$
By assuming that ansatz, one formally arrives to the following system of ODEs, which is called as the N vortex problem
$$
z_i'(t)=\frac{1}{2\pi}\sum_{j\neq i}\varpi_j\frac{(z_i(t)-z_j(t))^\perp}{|z_i(t)-z_j(t)|^2},
$$
see \cite{newton:book}. Such system of ODEs is a Hamiltonian system with the following Hamiltonian function
$$
H(z_1, \dots, z_N)=\frac{1}{4\pi}\sum_{i\neq j}\varpi_i\varpi_j \ln(|z_i(t)-z_j(t)|),
$$
which is directly coming from the stream function.

Then, in order to construct periodic disjoint vortex patches, one may use periodic evolutions of point vortices and {\it desingularize} them into patch solutions, that is, into proper solutions of the 2D Euler equations belonging to the Yudovich class.

In particular, one can verify that when $N=2$, namely when there are two point vortices, their dynamics is either a rigid rotation or a uniform translation. In Section~\ref{sec:mult}, we will perform the desingularization of two point vortices with equal vorticity strength that rotate around their center of mass. This provides the simplest example of a periodic evolution involving disjoint vortex patches, and was first carried out by Hmidi and Mateu in \cite{HM17}. Similar techniques have also been used to establish the existence of other classical vortex configurations, such as the Thomson polygon \cite{G21}, the K\'arm\'an vortex street \cite{G20}, asymmetric pairs \cite{HH21} or more general configurations \cite{HW22}. 
Finally, let us mention that the global continuation of the solutions constructed in \cite{HM17}, and later presented in Section \ref{sec:mult}, have been studied in \cite{GH22}.

\medskip\noindent
$\diamond$ {\it\sc Non rigid periodic vortex patches}

In the previous discussion, as well as in what follows in Sections \ref{sec:rigidpatch} and \ref{sec:mult}, we focus on rigid periodic motion. This significantly simplifies the equations under consideration, since the problem reduces to finding stationary solutions for the initial data in a rotating frame. However, there is strong evidence for the existence of non-rigid periodic solutions, and even quasi-periodic ones, close either to the Kirchhoff ellipse or to suitable configurations of four point vortices.

In recent years, several works have begun to investigate such structures using KAM techniques together with Nash--Moser iteration schemes, which are required due to the presence of small divisors. To connect with the two topics discussed above, we first address the existence of quasi-periodic solutions close to a single patch (namely, the Kirchhoff ellipse), and later the existence of desingularized non-rigid periodic patches near configurations of four point vortices.

The first work related to this direction is that of Berti--Hassainia--Masmoudi \cite{BHM22}, where the authors study the existence of quasi-periodic solutions close to the Kirchhoff ellipse
\[
D_\gamma := \left\{ (x,y) \in \mathbb{R}^2 : \frac{x^2}{\gamma} + \gamma y^2 = 1 \right\}, 
\qquad \gamma > 1,
\]
which is a rotating solution with angular velocity
\[
\Omega_\gamma := \frac{\gamma}{(1+\gamma)^2}.
\]
The goal of that work is to construct quasi-periodic vortex patches close to this configuration. Recall that a function $f(t)$ taking values in a Banach space $E$ is said to be quasi-periodic if
\[
f(t) = F(\varpi t),
\]
where $\varpi \in \mathbb{R}^\nu$ is a non-resonant frequency vector, that is,
\[
\varpi \cdot \ell \neq 0
\qquad \text{for every } \ell \in \mathbb{Z}^\nu \setminus \{0\}.
\]
The main theorem in \cite{BHM22} reads as follows.
\begin{theo}[Berti--Hassainia--Masmoudi \cite{BHM22}]
Let $\gamma_1<\gamma_2$. For any $\nu\in\N$, there exists a set $\mathcal{C}\subset[\gamma_1,\gamma_2]$ of asymptotically full Lebesgue measure such that, for any $\gamma\in\mathcal{C}$ there exists $\Omega$ close to $\Omega_\gamma$ and, in the rotating frame with angular velocity $\Omega$, a time quasi-periodic vortex patch solution $\omega={\bf 1}_{D(t)}$, with some frequency $\varpi$.
\end{theo}
Roughly speaking, the authors reduce the problem to finding nontrivial zeros of a functional depending on the perturbation of the angular velocity $\Omega$, the frequency vector $\omega$, and the parametrization of the perturbed domain $D$. However, since this nonlinear equation depends on both time and space, small divisors appear when attempting to invert the associated linearized operator, which leads to a loss of derivatives. For this reason, the authors implement a Nash--Moser scheme and employ tools from KAM theory to overcome these difficulties and invert the linear operator. For more details, we refer the reader to \cite{BHM22}.

The second work \cite{HHM23} by Hassainia--Hmidi--Masmoudi is motivated by the non-rigid periodic evolution of four point vortices in the complex plane. The first pair $z_1$ and $z_2$ are located in the upper half-plane with the same circulation $1$, while the second pair is located at $\overline{z}_1$ and $\overline{z}_2$ with the same negative circulation $-1$. Following the classical work \cite{HM17} discussed in Section \ref{sec:mult}, they desingularize these point vortices as
\[
\omega_{0,\varepsilon}(t,x)
=
\frac{1}{\pi \varepsilon^2}{\bf 1}_{z_1(t)+\varepsilon D_{t,1}^\varepsilon}(x)
+\frac{1}{\pi \varepsilon^2}{\bf 1}_{z_2(t)+\varepsilon D_{t,2}^\varepsilon}(x)
-\frac{1}{\pi \varepsilon^2}{\bf 1}_{\overline{z_1(t)}+\varepsilon \overline{D_{t,1}^\varepsilon}}(x)
-\frac{1}{\pi \varepsilon^2}{\bf 1}_{\overline{z_2(t)}+\varepsilon \overline{D_{t,2}^\varepsilon}}(x),
\]
for $\varepsilon\in(0,1)$, where $D_{t,k}^\varepsilon$ are domains localized around the unit disc.

They first study the periodic evolution of the four point vortices and show that the dynamics is periodic with period $T(\xi_0)$ provided that
$
\frac{\xi_0}{y_0}<\frac{\sqrt{2}}{2},
$
where $\xi_0$ denotes the distance between the two points $z_1$ and $z_2$, and $y_0$ is related to the initial center of mass. They then construct periodic vortex patches with the same period.

The proof again relies on finding nontrivial zeros of a functional depending on $\varepsilon$ and the parametrization of the domains $D_{t,k}^\varepsilon$. However, the equation becomes degenerate at $\varepsilon=0$, and due to the appearance of small divisors, the authors need to implement a Nash--Moser scheme together with tools from KAM theory in order to invert the linearized operator. Their main result reads as follows.

\begin{theo}[Hassainia--Hmidi--Masmoudi \cite{HHM23}]
Let $y_0>0$ and $0<\xi_1<\xi_2<\frac{y_0}{\sqrt{2}}$. There exists $\varepsilon_0$ such that for any $\varepsilon\in(0,\varepsilon_0)$ there exists a Cantor set $\mathcal{C}_\varepsilon\subset[\xi_1,\xi_2]$ of asymptotically full Lebesgue measure such that, for any $\xi_0\in \mathcal{C}_\varepsilon$, there exist $T(\xi_0)$--periodic domains $D_{k,\varepsilon}^t$ which solve the $2$D Euler equations.
\end{theo}

Let us also refer the reader to other related works on the two-dimensional Euler equations and other two-dimensional systems \cite{R21,R22,HR22,HR21,HHR23,GSIP23,HHM21}. Finally, let us mention the recent work \cite{GHH26} concerning leapfrogging motion for the $3$D Euler equations.

\section{Rigid periodic simply connected vortex patches}\label{sec:rigidpatch}

In this section we review the proof of Burbea \cite{B82} and Hmidi--Mateu--Verdera \cite{HMV13} concerning the existence of rotating vortex patches close to the Rankine vortex (the circular patch). The approach presented here is inspired by both works; however, in contrast with \cite{B82,HMV13}, we do not use conformal mappings to parametrize the boundary of the patch.

Our goal is to prove the existence of an angular velocity $\Omega\in\mathbb{R}$ and a bounded domain $D\subset\mathbb{R}^2$ such that the associated vortex patch is a rotating solution with angular velocity $\Omega$, that is:
$$
\omega_0={\bf 1}_D,\quad \omega(t,\cdot)={\bf 1}_{R_{\Omega t}D}.
$$
Since we seek rotating patch solutions, we can make use of equation \eqref{eq:rotating}, derived in Section~\ref{sec:intro}:
\begin{equation}\label{sec2-eq}
(u_0(x)-\Omega x^\perp)\cdot n(x)=0, \quad x\in\partial D,
\end{equation}
where $n$ is a normal vector to $\partial D$. Since we aim to apply the Crandall--Rabinowitz theorem, we look for solutions close to a trivial one, which in our setting is the Rankine vortex corresponding to the unit disc $D=\D$. We therefore parametrize the domain $D$ so that it can be viewed as a perturbation of the disc.  Motivated by recent works on quasi-periodic patches and non-rigid periodic patches, see \cite{BHM22, HHM23} and references therein, and in contrast with the conformal mapping approach used in \cite{B82,HMV13}, we adopt the following parametrization of the boundary $\partial D$:
$$
z(\theta)=\sqrt{1+2 f(\theta)}e^{i\theta},\quad w(\theta):=\sqrt{1+2 f(\theta)}.
$$ 
Such a parametrization is motivated by the fact that it preserves the area of the patch whenever $f$ has zero average. We should consider $f$ as a small perturbation, in particular, we assume $||f||_{L^\infty}< \frac12$ in order to have $1+2f(\theta)>0$. 
Notice now that
$$
n(x)=iw^{-1}(\theta)f'(\theta) e^{i\theta}-w(\theta) e^{i\theta},
$$
and then
$$
x^\perp \cdot n(x)=iw(\theta)e^{i\theta}\cdot \{iw^{-1}(\theta)f'(\theta) e^{i\theta}-w(\theta) e^{i\theta}\}=f'(\theta).
$$
Using that $u_0=\nabla^\perp\psi_0$ and inserting such parametrization in \eqref{sec2-eq} we arrive at:
$$
F(\Omega,f)(\theta):=\Omega f'(\theta)-\frac{\partial_\theta}{2\pi}\int_0^{2\pi}\int_0^{w(\eta)}\ln(|w(\theta) e^{i\theta}-r e^{i\eta}|)r dr d\eta=0.
$$
Now, the problem reduces to finding nontrivial zeros of the previous nonlinear and nonlocal functional $F$, which puts us in the appropriate framework to apply bifurcation techniques. Before stating our main result, let us introduce the function spaces that will be used throughout this section. Take $\alpha\in(0,1)$ and define the following H\"older spaces:
\begin{align}\label{funcionspaces}
    X^{k,\alpha}_m:=\{ f\in C^{k,\alpha}(\T), \quad f(\theta)=\sum_{n\in m\N^\star}f_n \cos(n\theta)\},\\
     Y^{k,\alpha}_m:=\{ f\in C^{k,\alpha}(\T), \quad f(\theta)=\sum_{n\in m\N^\star}f_n \sin(n\theta)\},
\end{align}
where the parameter $m$ stands for the $m$-fold symmetry of the domain.

Then, the main result stated in \cite{B82,HMV13} reads as follows.
\begin{theo}[Burbea \cite{B82}, Hmidi--Mateu--Verdera \cite{HMV13}]\label{th:burbea}
Let $m\geq 2$. There exists $\varepsilon>0$, and a map $\xi\in (-\varepsilon,\varepsilon)\mapsto (\Omega_\xi, f_\xi)\in \R\times X_m^{1,\alpha}$ such that
$$
F(\Omega_\xi, f_\xi)=0, \quad \forall \xi\in (-\varepsilon,\varepsilon).
$$
Furthermore $\Omega_\xi=\frac{m-1}{2m}+O(\xi)$ and $f_\xi=\xi \cos(m\theta)+o(\xi).$
\end{theo}
\begin{proof}
    The proof is an application of the Crandall-Rabinowitz theorem stated in Theorem \ref{th:CR}. Then, we shall check that all the hypotheses are satisfied:

\medskip \noindent  $\diamond$ Hypothesis (i):
  \medskip

  Note that
    \begin{align*}
        F(\Omega,0)=-\frac{\partial_\theta}{2\pi}\int_0^{2\pi}\int_0^1 \ln(|e^{i\theta}-r e^{i\eta}|)rdr d\eta.
    \end{align*}
    The modulus can be written as
    $$
    |e^{i\theta}-r e^{i\eta}|^2=1+r^2-2r\cos(\theta-\eta),
    $$
    and then we can do the following change of variables in the integral
        \begin{align*}
        F(\Omega,0)=&-\frac{\partial_\theta}{4\pi}\int_0^{2\pi}\int_0^1 \ln(1+r^2-2r\cos(\theta-\eta))rdr d\eta\\
        =&-\frac{\partial_\theta}{4\pi}\int_0^{2\pi}\int_0^1 \ln(1+r^2-2r\cos(\eta))rdr d\eta\\
        =&0,
    \end{align*}
    which is clearly vanishing for any $\Omega$ due to $\partial_\theta$.

\medskip \noindent  $\diamond$ Hypothesis (ii): \medskip

Take $\varepsilon<1$ and define $B_{X^{1,\alpha}_m}(\varepsilon)$ to be the ball of $X^{1,\alpha}_m$ centered at 0 of radius $\varepsilon$. We define the functional $F$ using the following spaces:
$$
F:\R\times B_{X^{1,\alpha}_m}(\varepsilon)\rightarrow Y^{0,\alpha}_m.
$$
We shall check that it is well-defined and $C^1$. For simplicity, define
\begin{equation}\label{defI}
I(f)(\theta):=\int_0^{2\pi}\int_0^{w(\eta)}\ln(|w(\theta) e^{i\theta}-r e^{i\eta}|)r dr d\eta.
\end{equation}
We begin by checking the symmetry in the spaces. Consider that $f\in B_{X^{1,\alpha}_m}(\varepsilon) $, then $f$ verifies $f(-\theta)=f(\theta)$ and $f'(-\theta)=-f'(\theta)$. In order to check that $F$ can be written as a Fourier series in terms of sines, we just need to check that $F(\Omega,f)$ is odd:
\begin{align*}
F(\Omega,f)(-\theta)
=&-\Omega f'(\theta)-\frac{1}{2\pi}\partial_\theta (I(f))(-\theta).
\end{align*}
Let us check that $I$ is even, which will imply that $F$ is odd:
\begin{align*}
I(f)(-\theta)=&\int_0^{2\pi}\int_0^{w(\eta)}\ln(|w(\theta) e^{-i\theta}-r e^{i\eta}|)r dr d\eta=\int_0^{2\pi}\int_0^{w(\eta)}\ln(|w(\theta) e^{i\theta}-r e^{-i\eta}|)r dr d\eta.
\end{align*}
Now we do the change of variables $\eta\mapsto -\eta$ inside the integral to find
\begin{align*}
I(f)(-\theta)
=&\int_0^{2\pi}\int_0^{w(\eta)}\ln(|w(\theta) e^{i\theta}-r e^{i\eta}|)r dr d\eta=I(f)(\theta),
\end{align*}
and then $I$ is even, which implies that $\partial_\theta I$ is odd and then $F$ is odd.

Now, let us check the $m$-fold symmetry. From the space $X_m^{1,\alpha}$, one has that
$$
f(\theta+\tfrac{2\pi}{m})=f(\theta).
$$
Let us check the same symmetry for $F$:
\begin{align*}
F(\Omega,f)(\theta+\tfrac{2\pi}{m})
=&\Omega f'(\theta+\tfrac{2\pi}{m})-\frac{1}{2\pi}\partial_\theta (I(f))(\theta+\tfrac{2\pi}{m})\\
=&\Omega f'(\theta)-\frac{1}{2\pi}\partial_\theta (I(f)(\theta+\tfrac{2\pi}{m})).
\end{align*}
Let us now check $I$:
\begin{align*}
I(f)(\theta+\tfrac{2\pi}{m})=&\int_0^{2\pi}\int_0^{w(\eta)}\ln(|w(\theta+\tfrac{2\pi}{m}) e^{i(\theta+\tfrac{2\pi}{m})}-r e^{i\eta}|)r dr d\eta\\
=&\int_0^{2\pi}\int_0^{w(\eta)}\ln(|w(\theta) e^{i\theta}-r e^{i(\eta-\tfrac{2\pi}{m})}|)r dr d\eta\\
=&\int_0^{2\pi}\int_0^{w(\eta+\tfrac{2\pi}{m})}\ln(|w(\theta) e^{i\theta}-r e^{i\eta}|)r dr d\eta\\
=&I(f)(\theta),
\end{align*}
which ends the proof of the $m$-fold symmetry of $F$. It remains to check the regularity properties. The first term is smooth in $(\Omega,f)$ and it remains to check the regularity with respect to $f$ of the second term. Indeed, we shall check that $I(f)\in C^{1,\alpha}$ and $f\mapsto I'(f)$ is continuous.

To check that $I(f)\in C^{1,\alpha}$, let us write $I(f)$ in the following way:
$$
I(f)(\theta)=\int_0^{2\pi}\int_0^1\ln(|w(\theta) e^{i\theta}-r w(\eta) e^{i\eta}|)w^2(\eta) r dr d\eta.
$$
Next, take the derivative in $\theta$ and compute the expression for the modulus obtaining
\begin{equation}\label{Idtheta}
\partial_\theta I(f)(\theta)=w'(\theta)\int_0^{2\pi}\int_0^1K_1(f)(r,\theta,\eta)w^2(\eta) r dr d\eta+\int_0^{2\pi}\int_0^1K_2(f)(r,\theta,\eta)w^2(\eta) r dr d\eta,
\end{equation}
with
\begin{align*}
&K_1(f)(r,\theta,\eta):=\\
&\frac{(w(\theta)-rw(\eta))+2rw(\eta)\sin^2((\theta-\eta)/2)+4rw(\theta)w(\eta)\sin((\theta-\eta)/2)\cos((\theta-\eta)/2)}{(w(\theta)-r w(\eta))^2+4rw(\theta)w(\eta)\sin^2((\theta-\eta)/2)},\\
&K_2(f)(r,\theta,\eta):=\frac{4rw(\theta)w(\eta)\sin((\theta-\eta)/2)\cos((\theta-\eta)/2)}{(w(\theta)-r w(\eta))^2+4rw(\theta)w(\eta)\sin^2((\theta-\eta)/2)}.
\end{align*}
In order to prove that $\partial_\theta I(f)\in C^{\alpha}$ we use classical tools from potential theory. For the completeness of this survey, we provide in Appendix \ref{sec:appendix} a lemma concerning the type of integral operators we will be dealing with, together with its estimate in H\"older spaces and the proof.  Then, we need to estimate $K$ in order to use Lemma \ref{lem:potentialtheory}. Let us first estimate the denominator. Note that
\begin{align*}
 ((w(\theta)-r w(\eta))^2+&4rw(\theta)w(\eta)\sin^2((\theta-\eta)/2))^{\frac12} 
 \geq C(|w(\theta)-r w(\eta)|+\sqrt{r}|\sin((\theta-\eta)/2)|),
\end{align*}
where $C$ denotes a constant that may change from line to line. Now, note that
\begin{align*}
    |w(\theta)-r w(\eta)|=&|(1-r)w(\theta)+r(w(\theta)- w(\eta))|
    \geq C(|1-r|-r\varepsilon|\sin((\theta-\eta)/2)|),
\end{align*}
or
\begin{align*}
    |w(\theta)-r w(\eta)|=&|(w(\theta)-w(\eta))+(1-r)w(\eta)|
    \geq C(|1-r|-\varepsilon|\sin((\theta-\eta)/2)|).
\end{align*}
By interpolating both inequalities we find
\begin{align*}
    |w(\theta)-r w(\eta)|
    \geq &C(|1-r|-\varepsilon \sqrt{r} |\sin((\theta-\eta)/2)|).
\end{align*}
Now, we come back to the denominator of $K$ to find
\begin{align*}
 ((w(\theta)-r w(\eta))^2+&4rw(\theta)w(\eta)\sin^2((\theta-\eta)/2))^{\frac12} 
 \geq C(|w(\theta)-r w(\eta)|+\sqrt{r}|\sin((\theta-\eta)/2)|)\\
 \geq &C(|1-r|+\sqrt{r}(1-\varepsilon)|\sin((\theta-\eta)/2)|)\\
 \geq &C(|1-r|+\sqrt{r}|\sin((\theta-\eta)/2)|).
\end{align*}
Using
$$
|w(\theta)-r w(\eta)|\leq C(|1-r|+r|\sin((\theta-\eta)/2)|),
$$
then, note that $K_1$ can be bounded by
\begin{align*}
|K_1(r,\theta,\eta)|\leq& C \frac{|1-r|+r|\sin((\theta-\eta)/2)|}{|1-r|^2+r\sin^2((\theta-\eta)/2)}\\
\leq& C\frac{1}{|1-r|+\sqrt{r}|\sin((\theta-\eta)/2)|}\\
\leq& \frac{C}{|1-r|^{1-\alpha} r^\frac{\alpha}{2}|\sin((\theta-\eta)/2)|^\alpha},
\end{align*}
where the final $C$ depends on $||f||_{X^{1,\alpha}_m}$, and we have interpolated with $\alpha\in(0,1)$. Similar estimates yield
\begin{align*}
|\partial_\theta K_1(r,\theta,\eta)|\leq& C\frac{1}{(|1-r|+\sqrt{r}|\sin((\theta-\eta)/2)|)^2}\leq C\frac{1}{|1-r|^\alpha r^{\frac{2-\alpha}{2}}|\sin((\theta-\eta)/2)|^{2-\alpha}},
\end{align*}
with $\alpha\in (0,1)$. By similar estimates for $K_2$, we can apply Lemma \ref{lem:potentialtheory} to have 
$$
||\partial_\theta I(f)||_{C^\alpha}\leq C||w^2||\leq C,
$$
obtaining that $I(f)\in C^{1,\alpha}$, and then $F$ is well-defined.

Let us now check that the Fr\'echet derivative is continuous. For that, we can directly study $\partial_\theta \partial_f I$, and we can interchange the derivatives and use the expression \eqref{Idtheta}:
\begin{align*}
\partial_f \partial_\theta I(f)h(\theta)=2\int_0^{2\pi}\int_0^1K(f)(r,\theta,\eta)h(\eta) r dr d\eta+\int_0^{2\pi}\int_0^1\partial_f K(f)h(r,\theta,\eta)w^2(\eta) r dr d\eta,
\end{align*}
where $K:=\partial_\theta w K_1+K_2$. In order to check the continuity in $f$ we can prove that there exists $\gamma>0$ such that
$$
||\partial_f \partial_\theta I(f_1)h-\partial_f \partial_\theta I(f_2)h||_{C^\alpha}\leq C||h||_{C^{1,\alpha}}||f_1-f_2||_{C^{1,\alpha}}^\gamma,
$$
which comes from studying each term separately and use Lemma \ref{lem:potentialtheory}, and $\gamma$ measures the modulus of continuity. To illustrate it, let us just take one the terms, for instance:
$$
I_1(f)=\int_0^{2\pi}\int_0^1\frac{(w(\theta)-rw(\eta))}{(w(\theta)-r w(\eta))^2+4rw(\theta)w(\eta)\sin^2((\theta-\eta)/2)}h(\eta) r dr d\eta,
$$
and
we aim to study
$$
I_1(f_1)-I_1(f_2)=\int_0^{2\pi}\int_0^1K_3(f_1,f_2)(r,\theta,\eta) h(\eta) r dr d\eta
$$
and 
\begin{align*}
K_3(f)(r,\theta,\eta):=&\frac{(w_1(\theta)-rw_1(\eta))}{(w_1(\theta)-r w_1(\eta))^2+4rw_1(\theta)w_1(\eta)\sin^2((\theta-\eta)/2)}\\
&-\frac{(w_2(\theta)-rw_2(\eta))}{(w_2(\theta)-r w_2(\eta))^2+4rw_2(\theta)w_2(\eta)\sin^2((\theta-\eta)/2)},
\end{align*}
where $w_j=\sqrt{1+2f_j}$. As we said, we aim to use Lemma \ref{lem:potentialtheory}, and for that we need to estimate $K_3$. For that we add and subtract appropriate terms as:

\begin{align*}
&K_3(f)(r,\theta,\eta):=\frac{(w_1(\theta)-rw_1(\eta)-w_2(\theta)-rw_2(\eta))}{(w_1(\theta)-r w_1(\eta))^2+4rw_1(\theta)w_1(\eta)\sin^2((\theta-\eta)/2)}\\
&+((w_2(\theta)-r w_2(\eta))^2-(w_1(\theta)-r w_1(\eta))^2+4r(w_2(\theta)w_2(\eta)-w_1(\theta)w_1(\eta))\sin^2((\theta-\eta)/2))\\
&\times\frac{(w_2(\theta)-rw_2(\eta))}{((w_2(\theta)-r w_2(\eta))^2+4rw_2(\theta)w_2(\eta)\sin^2((\theta-\eta)/2))}\\
&\times\frac{1}{((w_1(\theta)-r w_1(\eta))^2+4rw_1(\theta)w_1(\eta)\sin^2((\theta-\eta)/2))}.
\end{align*}
By using the estimates on the denominator, after several computations, we can estimate each term obtaining
\begin{align*}
|K_3(r,\theta,\eta)|
\leq& \frac{C||f_1-f_2||^\gamma_{C^{1,\alpha}}}{|1-r|^{\alpha} r^\frac{1-\alpha}{2}|\sin((\theta-\eta)/2)|^{1-\alpha}},\\
|\partial_\theta K_3(r,\theta,\eta)|
\leq& \frac{C||f_1-f_2||^\gamma_{C^{1,\alpha}}}{|1-r|^{\alpha} r^\frac{2-\alpha}{2}|\sin((\theta-\eta)/2)|^{2-\alpha}},
\end{align*}
for some $\gamma>0$ and any $\alpha\in(0,1)$, and then using Lemma \ref{lem:potentialtheory} we obtain 
$$
||\partial_f \partial_\theta I_1(f_1)h-\partial_f \partial_\theta I_2(f_2)h||_{C^\alpha}\leq C||h||_{C^{1,\alpha}}||f_1-f_2||_{C^{1,\alpha}}^\gamma.
$$
Here we have just provided a sketch of how to check the $C^1$ regularity of the functional, but we refer the reader to \cite{HMV13, HM17, GHS20, GHM23} and the references therein for more details. Finally, notice that
$$
\partial_f \partial_\Omega F(\Omega,f)h(\theta)=h'(\theta),
$$
which is clearly continuous.

\medskip \noindent  $\diamond$ Hypothesis (iii): \medskip

Here, we move to the spectral properties of the operator. Using the expression of $I(f)$ in \eqref{defI} we have
$$
\partial_f F(\Omega,0)h(\theta)=\Omega h'(\theta)-\mathcal{H}[h](\theta)-\frac{\partial_\theta}{2\pi}\int_0^{2\pi}\int_0^1 \frac{(e^{i\theta}-r e^{i\eta})\cdot h(\theta) e^{i\theta}}{|e^{i\theta}-re^{i\eta}|^2}r dr d\eta,
$$
where $\mathcal{H}$ stands for the Hilbert transform defined as
\begin{equation}\label{hilberttransform}
\mathcal{H}[h](\theta):=\frac{\partial_\theta}{\pi}\int_0^{2\pi}\ln(|e^{i\theta}-e^{i\eta}|)h(\eta) d\eta,
\end{equation}
implying
$$
\mathcal{H}[\cos(n\theta)]=\sin(n\theta), \quad n\geq 1.
$$
The last integral can be computed explicitly obtaining
$$
\partial_f F(\Omega,0)h(\theta)=\left(\Omega-\frac12\right) h'(\theta)-\frac12 \mathcal{H}[h](\theta).
$$
If $\Omega\neq \frac12$ one has that the operator
$$
h\in X^{1,\alpha}_m\mapsto \left(\Omega-\frac12\right)h'(\theta)\in Y^{0,\alpha}_m,
$$
is an isomorphism. Moreover
$$
h\in X^{1,\alpha}_m\mapsto \mathcal{H}[h]\in Y^{0,\alpha}_m,
$$
is compact since $\mathcal{H}[h]\subset C^{1,\alpha}$, which is compactly embedded in $C^{0,\alpha}$. Now, we use that compact perturbations of isomorphisms give us a Fredholm operator of index 0. Thus, our operator is Fredholm of index 0, and we can easily study its kernel by using Fourier series. Take $h$ as 
$$
h(\theta)=\sum_{n\in m\N^\star}h_n\cos(n\theta),
$$
and then
$$
\partial_f F(\Omega,0)h(\theta)=-\sum_{n\in m\N^\star}h_n n\sin(n\theta)\left(\Omega-\frac{n-1}{2n}\right).
$$
By defining $\Omega_m=\frac{m-1}{2m}$ one has that
$$
\textnormal{Ker }(\partial_f F(\Omega_m,0))=<\cos(m\theta)>,
$$
and
$$
Y^{0,\alpha}_m/\textnormal{Range }(\partial_f F(\Omega_m,0))=<\sin(m\theta)>,
$$
since all the other multipliers are non zero when we restrict to the $m$-fold space.

\medskip \noindent  $\diamond$ Hypothesis (iv): \medskip

For the transversality condition, we just compute the mixed derivative as:
\begin{align*}
    \partial_\Omega \partial_f F(\Omega,0)h(\theta)=-\sum_{n\in m\N^\star}h_n n\sin(n\theta),
\end{align*}
and evaluate at $\Omega_m$ and the element of the kernel:
\begin{align*}
    \partial_\Omega \partial_f F(\Omega_m,0)[\cos(m\theta)](\theta)=- m\sin(m\theta),
\end{align*}
which is clearly not in the range of $\partial_f F(\Omega_m,0)$. Then, applying Crandall-Rabinowitz theorem we arrive to the proof of this theorem.
\end{proof}

\section{Rigid periodic disjoint vortex patches}\label{sec:mult}
In this section we study a periodic evolution generated by two interacting vortex patches. More precisely, we review the work of Hmidi and Mateu \cite{HM17} on the existence of corotating vortex pairs. We present a simplified proof that avoids the use of conformal mappings and incorporates new ideas developed in \cite{HW22} to handle certain degeneracies in the linearized operator. Similar techniques have also been used in related works concerning the existence of other vortex configurations, such as the Thomson polygon, the K\'arm\'an vortex street, asymmetric vortex pairs and more general configurations \cite{G21,G20,HW22, HH21}.

Here, we aim to find $\varepsilon>0$, $\Omega\in\R$ and  $D \subset \R^2$ such that 
 $$
\omega_{0,\varepsilon}=\frac{1}{\pi \varepsilon^2}{\bf 1}_{P_1+\varepsilon D}+\frac{1}{\pi \varepsilon^2}{\bf 1}_{P_2-\varepsilon D},\quad \omega_\varepsilon(t,x)=\omega_{0,\varepsilon}(R_{-\Omega t} x).
$$
Here $P_1$ and $P_2$ represent the point vortices that we aim to desingularize, and
$$
P_2-\varepsilon D:=\{P_2-\varepsilon z, \quad z\in D\}.
$$
If $|D|=\pi$ we find that for small $\varepsilon$, the vorticity $\omega_{0,\varepsilon}$ converges in the distribution sense to two point vortices:
$$
\omega_{0,0}=\delta_{P_1}+\delta_{P_2}.
$$
We will see later how the rigid motion of the point vortices naturally appears, since the desingularization procedure will be performed from $\Omega=\Omega_0$, where $\Omega_0$ denotes the angular velocity of the corresponding two point vortices. For simplicity, we assume that the points $P_j$ are initially located on the real axis. In that case,
$$
P_1=(l,0),\quad P_2=(-l,0), \quad l\in\R^+.
$$
Inserting our ansatz into \eqref{eq:rotating} we find the following system of equations:
\begin{align*}
    (u_{0,\varepsilon}(x)-\Omega x^\perp)\cdot n_1(x)=&0, \quad x\in\partial D_1,\\
    (u_{0,\varepsilon}(x)-\Omega x^\perp)\cdot n_2(x)=&0, \quad x\in\partial D_2,
\end{align*}
where $n_j$ stands for a normal vector to $\partial D_j$. This is a coupled nonlinear and nonlocal system since $u_{0,\varepsilon}$ depends on both domains. Now, we parametrize $\partial D$ as:
$$
z(\theta)=\sqrt{1+2\varepsilon f(\theta)}e^{i\theta},\quad w(\theta):=\sqrt{1+2\varepsilon f(\theta)}.
$$
Inserting such a parametrization into the system, we realize that the second equation is equivalent to the first one and it suffices to study only one of them (this is due to the symmetry we have imposed in the domains). Then, after some computations, we arrive to the following nonlinear equation for $(\varepsilon,\Omega, f)$:
\begin{align*}
F(\varepsilon,\Omega, f)(\theta):=&\Omega l \sin(\theta) w(\theta)-\Omega l\cos(\theta)\varepsilon w^{-1}(\theta) f'(\theta)-\Omega \varepsilon^2 f'(\theta)\\
&+\frac{\partial_\theta}{2\pi^2 \varepsilon}\int_0^{2\pi}\int_0^{w(\eta)}\ln(|w(\theta)e^{i\theta}-r e^{i\eta}|) rdr d\eta \\
&+\frac{\partial_\theta}{2\pi^2 \varepsilon}\int_0^{2\pi}\int_0^{w(\eta)}\ln(|2l+\varepsilon(w(\theta) e^{i\theta}+r e^{i\eta})|) r dr d\eta=0.
\end{align*}
Let us also define
\begin{equation}\label{functionspace-new}
 \tilde{X}^{k,\alpha}_1:=\{ f\in C^{k,\alpha}(\T), \quad f(\theta)=\sum_{n\geq 2}f_n \cos(n\theta)\},
\end{equation}
which is the subspace of $X^{k,\alpha}_1$ consisting of functions with zero first Fourier mode.
In the following, let us state the main theorem in \cite{HM17}.
\begin{theo}[Hmidi-Mateu \cite{HM17}]
There exists $\varepsilon_0$, such that for any $\varepsilon\in(0,\varepsilon_0)$ there is a map $\varepsilon\mapsto (\Omega_\varepsilon, f_\varepsilon)\in\R\times \tilde{X}^{1,\alpha}_1$ such that
$$
F(\varepsilon,\Omega_\varepsilon,f_\varepsilon)=0.
$$
Moreover, $\Omega_0=\frac{1}{4\pi l^2}$ stands for the evolution of the two point vortices located at $(l,0)$ and $(-l,0)$.
\end{theo}
\begin{proof}
The proof is based on an application of the Implicit Function Theorem stated in Theorem~\ref{th:IFT}.

The first issue we need to address is the apparent singularity at $\varepsilon=0$. A priori, the functional $F$ is singular at this point. However, since our perturbation is a disc at first order, we can exploit the symmetries of the disc to eliminate this singularity. To this end, we make use of the following identity:
$$
\ln(A+\varepsilon B)=\ln(A)+\varepsilon B\int_0^1\frac{ds}{A+s\varepsilon B}.
$$
First, we make a change of variable in the integrals to write it as
    \begin{align*}
F(\varepsilon,\Omega, f)(\theta)=&\Omega l \sin(\theta) w(\theta)-\Omega l\cos(\theta)\varepsilon w^{-1}(\theta) f'(\theta)-\Omega \varepsilon^2 f'(\theta)\\
&+\frac{\partial_\theta}{2\pi^2 \varepsilon}\int_0^{2\pi}\int_0^{1}\ln(|w(\theta)e^{i\theta}-r w(\eta) e^{i\eta}|)w^2(\eta) rdr d\eta \\
&+\frac{\partial_\theta}{2\pi^2 \varepsilon}\int_0^{2\pi}\int_0^{1}\ln(|2l+\varepsilon(w(\theta) e^{i\theta}+r w(\eta) e^{i\eta})|) w^2(\eta) r dr d\eta.
\end{align*}
Note that defining 
$$
G_1(\varepsilon,f)(r,\theta,\eta):=2(f(\theta)+r^2 f(\eta))-2r \Big[\frac{w(\theta)-1}{\varepsilon}+ \frac{w(\eta)-1}{\varepsilon}+\frac{w(\theta)-1}{\varepsilon}(w(\eta)-1)\Big]\cos(\theta-\eta),
$$
we can write the kernel in the first integral as
\begin{align*}
  \ln(|w(\theta)e^{i\theta}-r w(\eta) e^{i\eta}|^2)=&\ln(1+r^2-2r\cos(\theta-\eta)+\varepsilon G_1(\varepsilon,f)(r,\theta,\eta))\\
  =&\ln(1+r^2-2r\cos(\theta-\eta))\\
  &+\varepsilon G_1(\varepsilon,f)(r,\theta,\eta)\int_0^1 \frac{ds}{1+r^2-2r\cos(\theta-\eta)+s\varepsilon G_1(\varepsilon,f)(r,\theta,\eta)}\\
  =:&\ln(1+r^2-2r\cos(\theta-\eta))+\varepsilon\mathcal{G}_1(\varepsilon,f)(r,\theta,\eta).
\end{align*}
Similarly, we can treat the second integral. Denoting
$$
G_2(\varepsilon,f)(r,\theta,\eta):=4(l,0)\cdot (w(\theta) e^{i\theta}+rw(\eta) e^{i\eta})+\varepsilon |w(\theta) e^{i\theta}+rw(\eta) e^{i\eta}|^2,
$$
we can write
\begin{align*}
  \ln(|2l+\varepsilon(w(\theta)e^{i\theta}+r w(\eta) e^{i\eta})|^2)=&\ln(4l^2+\varepsilon G_2(\varepsilon,f)(r,\theta,\eta))\\
  =&\ln(4l^2)\\
  &+\varepsilon G_2(\varepsilon,f)(r,\theta,\eta)\int_0^1 \frac{ds}{4l^2+s\varepsilon G_2(\varepsilon,f)(r,\theta,\eta)}\\
  =:&\ln(4l^2)+\varepsilon\mathcal{G}_2(\varepsilon,f)(r,\theta,\eta).
\end{align*}
Then, $F$ can be written as
    \begin{align*}
F(\varepsilon,\Omega, f)(\theta):=&\Omega l \sin(\theta) w(\theta)-\Omega l\cos(\theta)\varepsilon w^{-1}(\theta) f'(\theta)-\Omega \varepsilon^2 f'(\theta)\\
&+\frac{\partial_\theta}{4\pi^2 \varepsilon}\int_0^{2\pi}\int_0^{1}\ln(1+r^2-2r\cos(\theta-\eta))w^2(\eta) rdr d\eta \\
&+\frac{\partial_\theta}{4\pi^2 \varepsilon}\int_0^{2\pi}\int_0^{1}\ln(4l^2) w^2(\eta) r dr d\eta\\
&+\frac{\partial_\theta}{4\pi^2 }\int_0^{2\pi}\int_0^{1}\mathcal{G}_1(\varepsilon,f)(r,\theta,\eta) w^2(\eta) rdr d\eta \\
&+\frac{\partial_\theta}{4\pi^2 }\int_0^{2\pi}\int_0^{1}\mathcal{G}_2(\varepsilon,f)(r,\theta,\eta) w^2(\eta) r dr d\eta.
\end{align*}
Now, we shall use the symmetry of the disc. Indeed note that
\begin{align*}
    \partial_\theta \int_0^{2\pi}\int_0^1 \ln(1+r^2-2r\cos(\theta-\eta)) rdr d\eta=& 0,\\
    \partial_\theta \int_0^{2\pi}\int_0^1 \ln(4l^2)w^2(\eta) rdr d\eta=& 0,
\end{align*}
which is crucial to cancel all the singular terms in $\varepsilon$. Then, $F$ reduces to:
    \begin{align*}
F(\varepsilon,\Omega, f)(\theta)=&\Omega l \sin(\theta) w(\theta)-\Omega l\cos(\theta)\varepsilon w^{-1}(\theta) f'(\theta)-\Omega \varepsilon^2 f'(\theta)\\
&+\frac{\partial_\theta}{2\pi^2}\int_0^{2\pi}\int_0^{1}\ln(1+r^2-2r\cos(\theta-\eta))f(\eta) rdr d\eta \\
&+\frac{\partial_\theta}{4\pi^2 }\int_0^{2\pi}\int_0^{1}\mathcal{G}_1(\varepsilon,f)(r,\theta,\eta) w^2(\eta) rdr d\eta \\
&+\frac{\partial_\theta}{4\pi^2 }\int_0^{2\pi}\int_0^{1}\mathcal{G}_2(\varepsilon,f)(r,\theta,\eta) w^2(\eta) r dr d\eta.
\end{align*}
Then, we have arrived to an expression of $F$ that is nonsingular at $\varepsilon=0$. Now, let us define $F$ using the function spaces defined in \eqref{funcionspaces} and \eqref{functionspace-new}, where we have eliminated the first Fourier mode from $X^{k,\alpha}_1$. For $\delta<1$, we define
$$
F:\R^2\times B_{\tilde{X}^{1,\alpha}_1}(\delta)\rightarrow Y^{0,\alpha}_1.
$$
Similar ideas as in Theorem \ref{th:burbea} show that $F$ is well-defined and $C^1$, which is based in an application of Lemma \ref{lem:potentialtheory}. To apply the Implicit Function theorem, we need a trivial root. For that, let us compute $F(0,\Omega,0)$:
\begin{align*}
    F(0,\Omega,0)=&\Omega l\sin(\theta)+\frac{\partial_\theta}{4\pi^2 }\int_0^{2\pi}\int_0^{1}\mathcal{G}_1(0,0)(r,\theta,\eta)  rdr d\eta \\
&+\frac{\partial_\theta}{4\pi^2 }\int_0^{2\pi}\int_0^{1}\mathcal{G}_2(0,0)(r,\theta,\eta)  r dr d\eta\\
=&\Omega l\sin(\theta)+\frac{\partial_\theta}{4\pi^2 l^2}\int_0^{2\pi}\int_0^{1} (l,0)\cdot (e^{i\theta}+re^{i\eta}) r dr d\eta\\
=&\Omega l\sin(\theta)+\frac{\partial_\theta}{4\pi^2 l}\cos(\theta)\int_0^{2\pi}\int_0^{1}  r dr d\eta\\
=&\Omega l\sin(\theta)-\frac{1}{4\pi l}\sin(\theta).
\end{align*}
By defining $\Omega_0:=\frac{1}{4\pi l^2}$ one has
$$
F(0,\Omega_0,0)=0,
$$
where $\Omega_0$ coincides with the angular velocity of the two point vortices.

In order to apply the Implicit Function theorem, we shall check that
$$
(\partial_\Omega +\partial_f) F(0,\Omega_0,0): \R\times \tilde{X}^{1,\alpha}_1\rightarrow Y^{0,\alpha}_1,
$$
is an isomorphism. We can compute the Fr\'echet derivative to find that
\begin{align*}
 (\partial_\Omega +\partial_f) F(0,\Omega_0,0)[\lambda,h](\theta)=\lambda l\sin(\theta)+\frac{1}{2\pi}(h'(\theta)+\mathcal{H}[h](\theta)),   
\end{align*}
where $\mathcal{H}$ stands for the Hilbert transform defined in \eqref{hilberttransform}. Here we can see why it is necessary to also perturb the angular velocity $\Omega$. Indeed, observe that
\[
(\partial_\theta+\mathcal{H})\cos(\theta)
=
-\sin(\theta)+\sin(\theta)=0,
\]
which produces a degeneracy at the first Fourier mode.

If we only consider the derivative $\partial_f F$, then the corresponding direction belongs to the kernel of the linearized operator, violating the isomorphism condition required by the Implicit Function Theorem. In \cite{HM17}, this difficulty is resolved by fixing $\Omega=\Omega(\varepsilon,f)$ as a function of $(\varepsilon,f)$ in such a way that the first Fourier mode is completely eliminated from the nonlinear equation $F$. However, the resulting expression for $\Omega$ is rather complicated.

In \cite{HW22}, the authors adopt a different approach: they treat $\Omega$ as an additional bifurcation parameter and consider the linearization $(\partial_\Omega+\partial_f)F$ instead of $\partial_f F$ alone. This significantly simplifies the computations and removes the degeneracy at the first Fourier mode, as can be seen by writing the linearized operator in Fourier series:
\begin{align*}
    (\partial_\Omega +\partial_f) F(0,\Omega_0,0)[\lambda,h](\theta)=\lambda l\sin(\theta)-\frac{1}{2\pi}\sum_{n\geq 2}h_n\sin(n\theta)(n-1),
\end{align*}
which is an isomorphism from $\R\times \tilde{X}^{1,\alpha}_1\rightarrow Y^{0,\alpha}_1$. Notice that to obtain an isomorphism, it is important to eliminate the first Fourier mode in $\tilde{X}^{1,\alpha}_1$, otherwise it will appear in the kernel. By applying the Implicit Function theorem, we conclude the proof.
\end{proof}

\appendix

\section{Potential theory}\label{sec:appendix} 
This  section is devoted to some results on  the continuity of a specific  operator with singular kernels, taking the form 
\begin{equation}\label{int-op}
\mathcal{T}(f)(x)=\int_0^1\int_0^1 K(r,x,y)f(y)dr dy,
\end{equation}
with $x\in[0,1]$ and the kernel  $K:[0,1]^3 \rightarrow \R$ is smooth away $x=y$ and some discrete points for $r$. The proof of the following lemma is classical, many versions of it can be found in the literature, see \cite{HMV13, GHM22-1, GHM23} and references therein. However, we give here the proof for the sake of clarification.

\begin{lem}\label{lem:potentialtheory}
Let $K:[0,1]^3\rightarrow \R$ satisfy
\begin{align}
|K(r,x,y)|&\leq \frac{C_0 }{|x-y|^{1-\alpha}}{g_1(r)},\label{prop-potentialtheory-h0}\\ 
|\partial_{x} K(r,x,y)| &\leq \frac{C_0 }{|x-y|^{2-\alpha}}{g_2(r)}\label{prop-potentialtheory-h2},
\end{align}
with $\alpha\in(0,1)$ and {$g_1,g_2\in L^1([0,1])$}. Then $\mathcal{T}:L^\infty([0,1])\rightarrow C^\alpha([0,1])$ is well-defined and
$$
\|\mathcal{T}(f)\|_{C^\alpha}\leq C C_0 \|f\|_{L^\infty},
$$
with $C$ a constant only depending on $||g_1||_{L^1}$ and $||g_2||_{L^1}$.
\end{lem}
\begin{rem}
  Let us remark that the choice of the interval $[0,1]$ is made solely for the sake of simplifying the computations. Similar results hold for arbitrary bounded intervals. It even works for the periodic case (as in this document) when $x=\theta\in\T$ and $y=\eta\in\T$. In that case, the singularity $|x-y|$ is replaced by $|\sin((\theta-\eta)/2)|$ since
  $$
  |\sin((\theta-\eta)/2)|^{-1}<C\Big\{|\theta-\eta|^{-1}+|\theta-\eta-2\pi|^{-1}+|\theta-\eta+2\pi|^{-1}\Big\}.
  $$
\end{rem}
\begin{proof}
The $L^\infty$ norm of $\mathcal{T}(f)$ can be estimated as
\begin{align*}
\left|\mathcal{T}(f)(x)\right|\leq &C\|f\|_{L^\infty}\int_0^1\int_0^1 |K(r,x,y)|dr dy\\
\leq & C C_0\|f\|_{L^\infty}\int_0^1 \frac{dy}{|x-y|^{1-\alpha}}{\int_0^1|g_1(r)|dr}\\
\leq & C C_0\|f\|_{L^\infty}.
\end{align*}
The convergence follows from the assumptions $\alpha\in (0,1)$. Hence,
$$
\|\mathcal{T}(f)\|_{L^\infty}\leq C C_0\|f\|_{L^\infty}.
$$
For the H\"older regularity, take $x_1, x_2\in[0,1]$ with $x_1<x_2$. Define $d=|x_1-x_2|$, $B_{x_1}(\rho)=\left\{y\in[0,1] :  |y-x_1|<\rho\right\}$ and $B^c_{x_1}(\rho)$ its complement set.  Hence
\begin{align*}
\mathcal{T}(f)(x_1)&-\mathcal{T}(f)(x_2)\\
=&\int_0^1\int_0^1K(r,x_1,y)f(y)dr dy-\int_0^1\int_0^1K(r,x_2,y)f(y)dr dy\\
=&\int_0^1\int_{[0,1]\cap B_{x_1}(3d)}K(r,x_1,y)f(y)dr dy\\
&-\int_0^1\int_{[0,1]\cap B_{x_1}(3d)}K(r,x_2,y)f(y)dr dy\\
&+\int_0^1\int_{[0,1]\cap B^c_{x_1}(3d)}(K(r,x_1,y)-K(r,x_2,y))f(y)dr dy\\
=:& I_1+I_2+I_3.
\end{align*}
Using \eqref{prop-potentialtheory-h0}, we arrive at
\begin{align*}
|I_1|\leq& CC_0\|f\|_{L^\infty}\int_{[0,1]\cap B_{x_1}(3d)}\frac{1}{|x_1-y|^{1-\alpha}}dy\int_0^1{|g_1(r)|}dr \\
\leq& CC_0\|f\|_{L^\infty}\int_{B_{x_1}(3d)}\frac{1}{|x_1-y|^{1-\alpha}}dy\\
\leq & CC_0\|f\|_{L^\infty} d^\alpha\\
=&CC_0\|f\|_{L^\infty} |x_1-x_2|^\alpha.
\end{align*}
In order to work with $I_2$, note that $B_{x_1}(3d)\subset B_{x_2}(4d)$. Thus,
\begin{align*}
|I_2|\leq& CC_0\|f\|_{L^\infty}\int_{[0,1]\cap B_{x_1}(3d)}\frac{1}{|x_2-y|^{1-\alpha}}dy\int_0^1{|g_1(r)|}dr \\
\leq& CC_0\|f\|_{L^\infty}\int_{B_{x_2}(4d)}\frac{1}{|x_2-y|^{1-\alpha}}dy\\
\leq&CC_0\|f\|_{L^\infty} |x_1-x_2|^\alpha.
\end{align*}
For the last term $I_3$ we use the mean value theorem and \eqref{prop-potentialtheory-h2} achieving
\begin{align*}
|I_3|\leq &C \left|(x_1-x_2)\int_0^1\int_0^1\int_{[0,1]\cap B^c_{x_1}(3d)}(\partial_{x}K)(r, x_1+(1-s)(x_2-x_1),y)f(y)dydr ds\right|\\
\leq &C C_0 \|f\|_{L^\infty}|x_1-x_2| \int_0^1\int_{[0,1]\cap B^c_{x_1}(3d)}\frac{dyds}{|x_1+(1-s)(x_2-x_1)-y|^{2-\alpha}}\int_0^1{|g_2(r)|dr }.
\end{align*}
Note that if $y\in B^c_{x_1}(3d)$, then
$$
|x_1+(1-s)(x_2-x_1)-y|\geq |x_1-y|-(1-s)d\geq |x_1-y|-\frac{(1-s)}{3}|x_1-y|\geq \frac23|x_1-{y}|,
$$
which implies
\begin{align*}
|I_3|\leq &C C_0 \|f\|_{L^\infty} |x_1-x_2|\int_{[0,1]\cap B^c_{x_1}(3d)}\frac{dy}{|x_1-y|^{2-\alpha}}\\
\leq &C C_0 \|f\|_{L^\infty} |x_1-x_2|\frac{1}{|x_1-x_2|^{1-\alpha}}\\
\leq &C C_0 \|f\|_{L^\infty} |x_1-x_2|^\alpha.
\end{align*}
Putting together the preceding estimates yields
$$
|\mathcal{T}(f)(x_1)-\mathcal{T}(f)(x_2)|\leq C C_0 \|f\|_{L^\infty} |x_1-x_2|^\alpha.
$$
Then, 
$$
||\mathcal{T}(f)||_{\mathscr{C}^\alpha}\leq C C_0 \|f\|_{L^\infty},
$$
concluding the proof.
\end{proof}

\section*{Funding and conflicts of interest}

C.G. has been supported by RYC2022-035967-I (MCIU/AEI/10.13039/501100011033 and
FSE+), and partially by Grants PID2022-140494NA-I00 and PID2022-137228OB-I00 funded
by MCIN/AEI/10.13039/501100011033/FEDER, UE, by Grant C-EXP-265-UGR23 funded by
Consejeria de Universidad, Investigacion e Innovacion \& ERDF/EU Andalusia Program, and
by Modeling Nature Research Unit, project QUAL21-011. Proyecto realizado con la Beca Leonardo a
Investigadores y Creadores Culturales 2024 de la Fundaci\'on BBVA. 

The author has no competing interests to declare that are relevant to the content of this article.

	\bibliography{references}	
	\bibliographystyle{plain}

\end{document}